\date{}
\documentclass{article}
\usepackage{amsfonts}
\usepackage{amsmath}
\usepackage{amssymb}
\usepackage{graphicx}%
\usepackage{epstopdf}
\usepackage{amsthm}
\providecommand{\U}[1]{\protect \rule{.1in}{.1in}}
\newtheorem{theorem}{Theorem}

\newtheorem{definition}[theorem]{Definition}

\newtheorem{lemma}[theorem]{Lemma}
\newtheorem{notation}[theorem]{Notation}

\newtheorem{remark}[theorem]{Remark}

\begin{document}

\title{Basic Properties of \\Singular Fractional Order System with order (1,2)}
\author{Xiaogang Zhu, Jie Xu and Junguo Lu
\thanks{Jie Xu and Junguo Lu are with the School of Electronic Information
and Electrical Engineering, Shanghai Jiao Tong University, Shanghai,
200240 China }
}
\maketitle

\begin{abstract}
This paper focuses on some properties, which include regularity, impulse, stability, admissibility and robust admissibility, of singular fractional order system (SFOS)
with fractional order $1<\alpha<2$. The definitions of regularity, impulse-free, stability and admissibility are given in the paper. Regularity is analysed in time domain
 and the analysis of impulse-free is based on state response. A sufficient and necessary condition of stability is established.
 Three different sufficient and necessary conditions of admissibility are proved. Then, this paper shows how to get the
  numerical solution of SFOS in time domain. Finally, a numerical example is provided to illustrate the proposed conditions.
\end{abstract}

%

\tableofcontents

\section{Introduction}

Fractional order systems can describe the real physical systems better than
integer order systems because the real objects are generally fractional. A lot
of systems have been studied via fractional order systems, such as wavelet
transform \cite{Unser2000Fractional}, viscoelastic systems
\cite{Rossikhin1997Application} and others
(\cite{Makris1991Fractional,BAGLEY1991Fractional,Gao2005Synchronization,Engheta1996fractional,Hilfer2000Applications}%
).

Singular systems have been widely studied in many fields
(\cite{Dai1989Singular,Lewis1986survey,Kaczorek2011Singulara}) because
singular systems can describe real physical systems more directly than
regular systems. However, very few researches have been studied on singular
fractional order systems (SFOS), most of which are about stability.
In \cite{NDoye2010Stabilization}, a sufficient and necessary condition for
regularity is given; Based on regularity and free impulse, this paper also
gives a sufficient condition for stability. Sufficient and necessary
conditions for regularity and admissibility with fractional order $0<\alpha<1$
are given in \cite{Yao2013Sufcient}, respectively. Some other papers study the
stability of SFOS via linear matrix inequality (LMI)
(\cite{Song2012Stabilization,Ji2014Asymptotical,Ji2015Stabilization}) and some
study the stability of SFOS via transforming the SFOS into normal ones
(\cite{NDoye2013Robust,Yin2015Robust}).

However, non of them prove the
regularity, free impulse and stability in time domain, which can prove these
properties more directly. Moreover, to the best of our knowledge, there exists
no research on free impulse and admissibility with fractional order
$1<\alpha<2$. Therefore, in this paper we give the sufficient and necessary
conditions of regularity, free impulse, stability and admissibility for SFOS
with fractional order $1<\alpha<2$, respectively.

This paper is organized as follows.

In section II, the definition of Caputo's fractional derivative and SFOS are recalled. And some useful
lemma are provided. In section III, regularity and impulse are analysed in time domain. In section IV,
sufficient and necessary conditions of stability and admissibility are proved, respectively. In section V,
sufficient conditions of robust admissibility are presented. Finally, in section VI, numerical solution and
example are illustrated. Conclusion will be given in section VII.

\begin{notation}
For a matrix $A$, its transpose and complex conjugate transpose are denoted by
$A^{T}$ and $A^{\ast}$, respectively.$\
\mathbb{C}
_{-}=\left \{  s\mid s\in%
\mathbb{C}
\text{, }\operatorname{Re}(s)<0\right \}  $. $Sym(A)$ denotes $A+A^{\ast}$.
Denote pair $(E_{I},A_{I})$ as the autonomous singular integer order system
(SIOS) $E_{I}\overset{.}{x}\left(  t\right)  =A_{I}x\left(  t\right)  $.
Denote triplet $(E,A,\alpha)$ as the autonomous SFOS $ED^{\alpha}x\left(
t\right)  =Ax\left(  t\right)  $. The notation
$\bullet$ stands for the symmetric component in matrix.
\end{notation}

\section{Preliminaries}

In this paper, we use the Caputo's fractional derivative, of which the Laplace
transform allows utilization of initial values. The Caputo's fractional
derivative is defined as \cite{Podlubny1999Fractional}%

\[
_{a}D_{t}^{\alpha}f(t)=\frac{1}{\Gamma(\alpha-n)}\int_{a}^{t}\frac
{f^{(n)}(\tau)d\tau}{(t-\tau)^{\alpha+1-n}}%
\]

where $n$ is an integer satisfying $0\leq n-1<\alpha<n$; $\Gamma(\cdot)$ is
the Gamma function which is defined as%

\[
\Gamma(z)=\int_{0}^{\infty}e^{-t}t^{z-1}dt
\]

In the following of the paper, $_{a}D_{t}^{\alpha}$ is denoted by $D^{\alpha}$.

A two-parameter function of the Mittag-Leffler type is defined as
\cite{Podlubny1999Fractional}%

\[
E_{\alpha,\beta}(z)=\underset{k=0}{\overset{\infty}{%
{\textstyle \sum}
}}\frac{z^{k}}{\Gamma(\alpha k+\beta)}%
\]

where $\alpha>0,\beta>0.$

And $\delta^{\left(  -\beta \right)  }(t)$ means%

\[
\delta^{\left(  -\beta \right)  }(t)=\left \{
\begin{array}
[c]{c}%
\frac{t^{\beta-1}}{\Gamma(\beta)}\\
0
\end{array}%
\begin{array}
[c]{c}%
t>0\\
t<0
\end{array}%
\begin{array}
[c]{c}%
\beta \in%
\mathbb{R}%
\end{array}
\right.
\]
whose Laplace transform is

\begin{align*}
\textit{L}[\delta^{-\alpha}(t)]=s^{-\alpha},\ \textrm{Re}(s)>0
\end{align*}

Consider the singular fractional order system (SFOS)%
\begin{equation}
\left \{
\begin{tabular}
[c]{c}%
$ED^{\alpha}x\left(  t\right)  =Ax\left(  t\right)  +Bu\left(  t\right)  $\\
$y\left(  t\right)  =Cx\left(  t\right)  +Du\left(  t\right)  $%
\end{tabular}
\  \right.  \label{SFOS}%
\end{equation}
where $x\left(  t\right)  \in \mathbb{R}^{n}$ is the state of the system
composed of state variables; $u(t)\in \mathbb{R}^{p}$ is the control input;
$y(t)\in \mathbb{R}^{q}$ is the measure output; $E,A\in%
\mathbb{R}
^{n\times n}$; $B,C,D$ are constant matrices with appropriate dimensions;
$D^{\alpha}$ represents the Caputo fractional derivative; $\alpha$ is the
order of the SFOS and $1<\alpha<2$.

The finite eigenvalues of SFOS is
\[ 
\lambda(E,A)=\{s\mid s\in\mathbb{C}
,\left \vert s\right \vert <\infty,\det(sE-A)=0\}\]
The finite pole set for the
system is 
\[\sigma(E,A)=\left \{  s\mid s\in%
\mathbb{C}
,\text{ }\left \vert s\right \vert <\infty \text{, }\det(s^{\alpha}E-A)=0\text{,
}0<\alpha<2\right \}\]
and $\sigma(I,A)$ will be specified as $\sigma(A)$.
Obiviously, $\lambda(E,A)=\sigma^{\alpha}(E,A)$.

The following lemmas and definitions will be useful.

\begin{lemma}
\label{lemGantmacher} \cite{Gantmacher1974theory} For any two matrices $E,A\in%
\mathbb{R}
^{m\times n}$, there always exist two nonsingular matrices $Q,P$ such that%
\begin{equation}%
\begin{tabular}
[c]{c}%
$\widetilde{E}\triangleq QEP=diag(\mathbf{0},L_{1},L_{2},...,L_{p}%
,L_{1}^{^{\prime}},L_{2}^{^{\prime}},...L_{q}^{^{\prime}},I,N)$\\
$\widetilde{A}\triangleq QAP=diag(\mathbf{0},J_{1},J_{2},...,J_{p}%
,J_{1}^{^{\prime}},J_{2}^{^{\prime}},...J_{q}^{^{\prime}},A_{1},I)$%
\end{tabular}
\  \label{equGantmacher}%
\end{equation}

where%
\[
\mathbf{0}\in%
\mathbb{R}
^{m_{0}\times n_{0}},A_{1}\in%
\mathbb{R}
^{h\times h}%
\]%
\[
L_{i}=\left[
\begin{array}
[c]{ccccc}%
1 & 0 &  &  & \\
& 1 & 0 &  & \\
&  & \ddots & \ddots & \\
&  &  & 1 & 0
\end{array}
\right]  ,J_{i}=\left[
\begin{array}
[c]{ccccc}%
0 & 1 &  &  & \\
& 0 & 1 &  & \\
&  & \ddots & \ddots & \\
&  &  & 0 & 1
\end{array}
\right]  \in%
\mathbb{R}
^{m_{i}\times(m_{i}+1)}%
\]

\rightline{$i=1,2,...,p$}%

\[
L_{j}^{^{\prime}}=\left[
\begin{array}
[c]{cccc}%
1 &  &  & \\
0 & 1 &  & \\
& 0 & \ddots & \\
&  & \ddots & 1\\
&  &  & 0
\end{array}
\right]  ,J_{j}^{^{\prime}}=\left[
\begin{array}
[c]{cccc}%
0 &  &  & \\
1 & 0 &  & \\
& 1 & \ddots & \\
&  & \ddots & 0\\
&  &  & 1
\end{array}
\right]  \in%
\mathbb{R}
^{(n_{j}+1)\times n_{j}}%
\]

\rightline{$j=1,2,...,q$}%

\[
N=diag(N_{k_{1}},N_{k_{2}},...,N_{k_{r}})\in%
\mathbb{R}
^{g\ast g}%
\]%
\[
N_{k_{s}}=\left[
\begin{array}
[c]{cccc}%
0 & 1 &  & \\
& 0 & \ddots & \\
&  & \ddots & 1\\
&  &  & 0
\end{array}
\right]  \in%
\mathbb{R}
^{k_{s}\times k_{s}},\ s=1,2,...,r
\]%
\[
m_{0}+\underset{i}{\sum}m_{i}+\underset{j}{\sum}(n_{j}+1)+\underset{s}{\sum
}k_{s}+h=m
\]%
\[
n_{0}+\underset{j}{\sum}n_{j}+\underset{i}{\sum}(m_{i}+1)+\underset{s}{\sum
}k_{s}+h=n
\]%
\[
\underset{s}{\sum}k_{s}=g
\]

\end{lemma}

Consider the following initial-value problem:%
\begin{equation}
_{0}D_{t}^{\sigma_{n}}y(t)+\overset{n-1}{\underset{j=1}{\sum}}p_{j}%
(t)_{0}D_{t}^{\sigma_{n-j}}y(t)+p_{n}(t)y(t)=f(t) \label{FOS_init1}%
\end{equation}%
\[
(0<t<T<\infty)
\]%
\begin{equation}
\left[  _{0}D_{t}^{\sigma_{k}-1}y(t)\right]  _{t=0}=b_{k},\ k=1,2,...,n
\label{FOS_init2}%
\end{equation}

where%
\[
_{a}D_{t}^{\sigma_{k}}\equiv{}_{a}D_{t}^{\alpha_{k}}{}_{a}D_{t}^{\alpha_{k-1}%
}...{}_{a}D_{t}^{\alpha_{1}}%
\]%
\[
_{a}D_{t}^{\sigma_{k}-1}\equiv{}_{a}D_{t}^{\alpha_{k}-1}{}_{a}D_{t}%
^{\alpha_{k-1}}...{}_{a}D_{t}^{\alpha_{1}}%
\]%
\[
\sigma_{k}=\underset{j=1}{\overset{k}{\sum}}\alpha_{j},\text{ }(k=1,2,...,n)
\]%
\[
0<\alpha_{j}\leq1,\text{ }(j=1,2,...,n)
\]

and $f(t)\in L_{1}(0,T)$, i.e.%
\[
\int_{0}^{T}\left \vert f(t)\right \vert dt<\infty
\]

\begin{lemma}
\label{lemFOS_init} \cite{Podlubny1999Fractional} If $f(t)\in L_{1}(0,T)$, and
$p_{j}(t)$ $(j=1,2,...,n)$ are continuous functions in the closed interval
$[0,T]$, then the initial-value problem (\ref{FOS_init1})-(\ref{FOS_init2})
has a unique solution $y(t)\in L_{1}(0,T)$.
\end{lemma}


\begin{definition}
\cite{Chilali1996H} A subset $\mathcal{D}$ of the complexplane is called an
LMI region if there exist a symmetric matrix $\Phi \in%
\mathbb{R}
^{d\times d}$ and a matrix $\Psi \in%
\mathbb{R}
^{d\times d}$ such that
\begin{equation}
\mathcal{D}=\left \{  z\in%
\mathbb{C}
\mid f_{\mathcal{D}}(z)<0\right \}  \label{DRegion}%
\end{equation}

where $f_{\mathcal{D}}(z)=$ $\Phi+z\Psi+\bar{z}\Psi^{T}$ and $"<"$ stands for
negative definite. When $\Phi=0$, the LMI region is denoted by $\mathcal{D}%
_{\Gamma}$.
\end{definition}

\begin{definition}
\cite{Chilali1996H} If all the eigenvalues of $A\in%
\mathbb{R}
^{n\times n}$ take values in region $\mathcal{D}$, i.e. $\lambda
(A)\subset \mathcal{D}$, then $A$ is called $\mathcal{D}$-stable.
\end{definition}

\begin{lemma}
\label{lemADstable}\cite{Chilali1996H} Matrix $A$ is $\mathcal{D}$-stable if
and only if there exists a symmetric real matrix $X>0$ such that
\[
M_{\mathcal{D}}\left(  A,X\right)  =\Phi \otimes X+\Psi \otimes \left(
XA\right)  +\Psi^{T}\otimes \left(  AX\right)  ^{T}<0
\]

\end{lemma}

\begin{lemma}
\label{lemNorStable} \cite{Moze2005LMI} System $D^{\alpha}x(t)=Ax(t)+Bu(t)$
with fractional order $1<\alpha<2$ is asymptotically stable if and only if
there there exists a matrix $P>0,P\in \mathbb{R}^{n\times n}$ such that
\begin{equation}
Sym\{ \Theta \otimes(AP)\}{}<0
\end{equation}
where $\Theta=\left[
\begin{array}
[c]{cc}%
\sin \frac{\pi}{2}\alpha & -\cos \frac{\pi}{2}\alpha \\
\cos \frac{\pi}{2}\alpha & \sin \frac{\pi}{2}\alpha
\end{array}
\right]  $
\end{lemma}

\begin{lemma}
\label{lemIneMa} \cite{Xu2006Robust} Let $X,Y,\Lambda$ be real matrices of
suitable dimensions and $\Lambda>0$, then
\begin{equation}
X^{T}Y+Y^{T}X\leq X^{T}\Lambda X+Y^{T}\Lambda Y
\end{equation}

\end{lemma}

\begin{definition}
For system $(E,A,\alpha)$, the infinite eigenvectors $\upsilon$, which are
related to eigenvalue 0, are defined as follows

(1) The infinite eigenvector of order 1 satisfies $E\upsilon^{1}%
=0,\  \upsilon^{0}=0;$

(2) The infinite eigenvector of order k satisfies $E\upsilon^{k}%
=A\upsilon^{k-1},\ k>1.$

\begin{remark}
The infinite eigenvector isn't related to the index $\alpha$, which implies it
may have the same properties as the infinite eigenvectors of SIOS.
\end{remark}
\end{definition}


\section{Solution of SFOS}

\subsection{Regularity of SFOS}

The sufficient and necessary condition of regularity for SFOS have already
been given in \cite{NDoye2010Stabilization,Yao2013Sufcient}. But the systems
in \cite{NDoye2010Stabilization,Yao2013Sufcient} are a linear SFOS and the
fractional order in \cite{Yao2013Sufcient} is $0<\alpha<1$. In the following, a
different definition of regularity is proposed. And based on this definition, we give a
sufficient and necessary condition of regularity for nonlinear SFOS with
fractional order $1<\alpha<2$.

Let $Bu\left(  t\right)  =g(t)$, then the system (\ref{SFOS}) can be rewrite
as
\begin{equation}
ED^{\alpha}x\left(  t\right)  =Ax\left(  t\right)  +g(t)
\label{SFOS-regularity}%
\end{equation}
where $g(t)$ is nonlinear and assumed to be sufficiently differential;
$1<\alpha<2$. We will focus on the existence, uniqueness of
(\ref{SFOS-regularity}).

\begin{definition}
If a SFOS has a unique solution, then the system is termed \textit{regular}.
\end{definition}

\begin{theorem}
\label{theoSFOSregular} System (\ref{SFOS-regularity}) is regular if and only
if two nonsingular matrices $Q$ and $P$ may be chosen such that%
\[%
\begin{array}
[c]{c}%
QEP=diag(I_{n_{1}},N)\\
QAP=diag(A_{1},I_{n_{2}})
\end{array}
\]

where $A_{1}\in%
\mathbb{R}
^{n_{1}\times n_{1}}$; $N\in%
\mathbb{R}
^{n_{2}\times n_{2}}$ is nilpotent; $n_{1}+n_{2}=n$.
\end{theorem}

\begin{proof}
According to lemma \ref{lemGantmacher}, let $x(t)=P\overset{\sim}{x}(t)$ and
left multiply system (\ref{SFOS-regularity}) by a nonsingular $Q$. Let
$\overset{\sim}{g}(t)=Qg(t)$, we get%
\begin{equation}
\mathbf{0}D^{\alpha}x_{n_{0}}(t)=g_{m_{0}}(t) \label{SFOS_Reg1}%
\end{equation}%
\begin{equation}
L_{i}D^{\alpha}x_{m_{i}+1}(t)=J_{i}x_{m_{i}+1}(t)+g_{m_{i}}(t),\text{
}i=1,2,...,p \label{SFOS_Reg2}%
\end{equation}%
\begin{equation}
L_{j}^{^{\prime}}D^{\alpha}x_{n_{j}}(t)=J_{j}^{^{\prime}}x_{n_{j}}%
(t)+g_{n_{j}+1}(t),\text{ }j=1,2,...,q \label{SFOS_Reg3}%
\end{equation}%
\begin{equation}
N_{k_{s}}D^{\alpha}x_{k_{s}}(t)=x_{k_{s}}(t)+g_{k_{s}}(t),\text{ }s=1,2,...,r
\label{SFOS_Reg4}%
\end{equation}%
\begin{equation}
D^{\alpha}x_{h}(t)=A_{1}x_{h}(t)+g_{h}(t) \label{SFOS_Reg5}%
\end{equation}

where $L_{i},L_{j}^{^{\prime}},J_{j},J_{j}^{^{\prime}}$ are defined in
(\ref{equGantmacher}) and$\ x_{k}(t)\in%
\mathbb{R}
^{k},g_{k}(t)\in%
\mathbb{R}
^{k},$%
\[
\overset{\sim}{x}^{T}(t)=\left[  x_{n_{0}}^{T},x_{m_{1}+1}^{T},\cdots
,x_{m_{p}+1}^{T},x_{n_{1}}^{T},\cdots,x_{n_{q}}^{T},x_{k_{1}}^{T}%
,\cdots,x_{k_{r}}^{T},x_{h}^{T}\right]
\]%
\[
\overset{\sim}{g}^{T}(t)=\left[  g_{m_{0}}^{T},g_{m_{1}}^{T},\cdots,g_{m_{p}%
}^{T},g_{n_{1}+1}^{T},\cdots,g_{n_{q}+1}^{T},g_{k_{1}}^{T},\cdots,g_{k_{r}%
}^{T},g_{h}^{T}\right]
\]

System (\ref{SFOS_Reg1})-(\ref{SFOS_Reg5}) is equivalent to system
(\ref{SFOS-regularity}), thus we focus on the existence, uniqueness of system
(\ref{SFOS_Reg1})-(\ref{SFOS_Reg5}).

(1) If equation (\ref{SFOS_Reg1}) can be solved, then $g_{m_{0}}(t)=0$ must be true. In
this case, equation (\ref{SFOS_Reg1}) is always true. Therefore, this equation
has either no solution or an infinite number of solutions.

(2) Equation (\ref{SFOS_Reg2}) is composed of a set of equations
\begin{equation}
\left \{
\begin{array}
[c]{c}%
D^{\alpha}z_{1}\left(  t\right)  =z_{2}\left(  t\right)  +g_{1}\left(
t\right) \\
D^{\alpha}z_{2}\left(  t\right)  =z_{3}\left(  t\right)  +g_{2}\left(
t\right) \\
\cdots \\
D^{\alpha}z_{k-1}\left(  t\right)  =z_{k}\left(  t\right)  +g_{k-1}\left(
t\right)
\end{array}
\  \right.  \label{SFOS_Reg2_2}%
\end{equation}

According to lemma \ref{lemFOS_init}, for a certain $z_{k}\left(  t\right)  ,$
$z_{1}\left(  t\right)  ,z_{2}\left(  t\right)  ,...,z_{k-1}\left(  t\right)
$ can be determined successively. Therefore, such equations have an infinite
number of solutions.

(3) Rewrite equation (\ref{SFOS_Reg3}) as%
\[
\left \{
\begin{array}
[c]{c}%
D^{\alpha}z_{1}\left(  t\right)  =g_{1}\left(  t\right) \\
D^{\alpha}z_{2}\left(  t\right)  =z_{1}\left(  t\right)  +g_{2}\left(
t\right) \\
\cdots \\
D^{\alpha}z_{k}\left(  t\right)  =z_{k-1}\left(  t\right)  +g_{k}\left(
t\right) \\
0=z_{k}\left(  t\right)  +g_{k+1}\left(  t\right)
\end{array}
\right.
\]

Except the last equation, $z_{1}\left(  t\right)  ,z_{2}\left(  t\right)
,...,z_{k}\left(  t\right)  $ can be determined uniquely according to lemma
\ref{lemFOS_init}. However, $z_{k}\left(  t\right)  $ must satisfy the last
euqation, which means these euqations have no solution unless $g_{k+1}\left(
t\right)  $ satisfies the consistent condition $z_{k}\left(  t\right)
+g_{k+1}\left(  t\right)  =0.$

(4) Expand equation (\ref{SFOS_Reg4}) into the following form%
\[
\left \{
\begin{array}
[c]{c}%
D^{\alpha}z_{2}\left(  t\right)  =z_{1}\left(  t\right)  +g_{1}\left(
t\right) \\
D^{\alpha}z_{3}\left(  t\right)  =z_{2}\left(  t\right)  +g_{2}\left(
t\right) \\
\cdots \\
D^{\alpha}z_{k}\left(  t\right)  =z_{k-1}\left(  t\right)  +g_{k-1}\left(
t\right) \\
0=z_{k}\left(  t\right)  +g_{k}\left(  t\right)
\end{array}
\right.
\]

Beginning with the last equation, $z_{1}\left(  t\right)  ,z_{2}\left(
t\right)  ,...,z_{k}\left(  t\right)  $ may be determined successively for
sufficiently differentiable functions $g_{i}\left(  t\right)  $
$(i=1,2,...,k).$ Therefore, equation (\ref{SFOS_Reg4}) has a unique solution.

(5) Equation (\ref{SFOS_Reg5}) is an ordinary fractional order differential
equation, which has a unique solution since $g(t)$ is sufficiently differential.

To sum up, the system (\ref{SFOS-regularity}) exists a solution and the
solution is unique if and only if two nonsingular matrices $Q$ and $P$ may be
chosen to satisfy%
\[
QEP=diag(I_{n_{1}},N)
\]%
\[
QAP=diag(A_{1},I_{n_{2}})
\]

where $N=diag(N_{k_{1}},N_{k_{2}},...,N_{k_{r}}).$ The theorem is proved.

\end{proof}

\subsection{State response and impulse analysis}

To the best of our knowledge, there exists no research which gives the entire
state response for SFOS. The following gives an entire state response, based
on which we give a sufficient and necessary condition of impulse-free for SFOS.

Consider the regular SFOS%
\begin{equation}
ED^{\alpha}x\left(  t\right)  =Ax\left(  t\right)  +Bu\left(  t\right)
\label{RegSFOS}%
\end{equation}

where $E\in%
\mathbb{R}
^{n\times n}$, $1<\alpha<2$ and the initial condition $x(0)=x_{0},$ $t\geq0$.

\begin{theorem}
\label{theoSFOSresponse} When $t\geq0$, the state response to SFOS
(\ref{RegSFOS}) is%
\begin{equation}
x\left(  t\right)  =P\left[
\begin{array}
[c]{c}%
x_{1}(t)\\
x_{2}(t)
\end{array}
\right]  \label{SFOS response}%
\end{equation}

where%
\[
x_{1}(t)=E_{\alpha,1}\left(  A_{1}t^{\alpha}\right)  x_{10}+%
{\displaystyle \int \limits_{t_{0}}^{t}}
(t-\tau)^{\alpha-1}E_{\alpha,\alpha}\left(  A_{1}(t-\tau)^{\alpha}\right)
B_{1}u\left(  \tau \right)  d\tau
\]%
\begin{align*}
x_{2}(t)  &  =-\underset{k=1}{\overset{h-1}{%
{\textstyle \sum}
}}N^{k}\left(  \delta^{\left(  k\alpha-1\right)  }\left(  t\right)
x_{20}+\delta^{\left(  k\alpha-2\right)  }\left(  t\right)  x_{20}^{\left(
1\right)  }\right) \\
&  -\underset{k=0}{\overset{h-1}{%
{\textstyle \sum}
}}N^{k}B_{2}\left(  D^{k\alpha}u\left(  t\right)  +\underset{j=0}%
{\overset{m-1}{%
{\textstyle \sum}
}}u_{0}^{(j)}\delta^{\left(  k\alpha-j-1\right)  }\left(  t\right)  \right)
\end{align*}

$x_{1}(t)\in%
\mathbb{R}
^{n_{1}},$ $x_{2}(t)\in%
\mathbb{R}
^{n_{2}},$ $n_{1}+n_{2}=n$, the initial condition $x_{1}(0)=x_{10}$,
$x_{2}(0)=x_{20}$, $\overset{.}{x}_{2}(0)=x_{20}^{\left(  1\right)  }$; $N\in%
\mathbb{R}
^{n_{2}\times n_{2}}$ is nilpotent and the nilpotent index is denoted by $h$;
$u(t)$ is $h$ times piecewise continuously differentiable, the initial
condition $u^{(j)}(0)=u_{0}^{(j)}$; $m$ is an integer and $m-1<k\alpha \leq m$.
$E_{\alpha,\beta}$ is the two-parameter function of the Mittag-Leffler type. $P$ satisfies
Theorem \ref{theoSFOSregular}.
When $t>0$ and the initial condition%
\begin{align*}
x\left(  0_{+}\right)   &  =P\left[
\begin{array}
[c]{c}%
I\\
0
\end{array}
\right]  x_{10}-P\left[
\begin{array}
[c]{c}%
0\\
I
\end{array}
\right]  \underset{k=1}{\overset{h-1}{%
{\textstyle \sum}
}}M^{-1}N^{k}\delta^{\left(  k\alpha-2\right)  }\left(  0_{+}\right)
x_{20}^{\left(  1\right)  }\\
&  -P\left[
\begin{array}
[c]{c}%
0\\
I
\end{array}
\right]  \underset{k=0}{\overset{h-1}{%
{\textstyle \sum}
}}M^{-1}N^{k}B_{2}\left(  D^{k\alpha}u\left(  0_{+}\right)  +\underset
{j=0}{\overset{m-1}{%
{\textstyle \sum}
}}u^{(j)}(0)\delta^{\left(  k\alpha-j-1\right)  }\left(  0_{+}\right)
\right)
\end{align*}

is satisfied, then the solution (\ref{SFOS response}) to system (\ref{RegSFOS}%
) is unique and $M=I+\underset{k=1}{\overset{h-1}{%
{\textstyle \sum}
}}N^{k}\delta^{\left(k\alpha -1 \right)  }\left(
0_{+}\right)  .$
\end{theorem}

\begin{proof}
Since the system is regular, two nonsingular matrices $P,Q$ may be chosen and
the system (\ref{RegSFOS}) is equivalent to%
\begin{equation}
D^{\alpha}x_{1}\left(  t\right)  =A_{1}x_{1}\left(  t\right)
+B_{1}u\left(  t\right)  \label{SFOSfinite}%
\end{equation}%
\begin{equation}
ND^{\alpha}x_{2}\left(  t\right)  =x_{2}\left(  t\right)
+B_{2}u\left(  t\right)  \label{SFOSinfinite}%
\end{equation}

where $QB=\left[  B_{1}^{T}\text{ }B_{2}^{T}\right]  ^{T}.$ Subsystems
(\ref{SFOSfinite}) and (\ref{SFOSinfinite}) are termed finite subsystem and
infinite subsystem, respectively, which are similar to SIOS.

Finite subsystem (\ref{SFOSfinite}) is an normal fractional order system. For
the piecewise continuously differentiable input $u(t),$ the state response to
the subsystem (\ref{SFOSfinite}) is
\begin{equation}
x_{1}(t)=E_{\alpha,1}\left(  A_{1}t^{\alpha}\right)  x_{10}+%
{\displaystyle \int \limits_{t_{0}}^{t}}
(t-\tau)^{\alpha-1}E_{\alpha,\alpha}\left(  A_{1}(t-\tau)^{\alpha}\right)
B_{1}u\left(  \tau \right)  d\tau \label{finSFOSsolu}%
\end{equation}

For the infinite subsystem (\ref{SFOSinfinite}), invoking Laplace transform
and $(sN-I)^{-1}=-\underset{k=0}{\overset{h-1}{\sum}}s^{k}N^{k}$%
\begin{align}
X_{2}(s)  &  =(s^{\alpha}N-I)^{-1}\left[  N\left(  s^{\alpha-1}x_{20}%
+s^{\alpha-2}x_{20}^{(1)}\right)  +B_{2}U(s)\right] \nonumber \\
&  =-\underset{k=1}{\overset{h-1}{\sum}}N^{k}\left(  s^{k\alpha-1}%
x_{20}+s^{k\alpha-2}x_{20}^{\left(  1\right)  }\right)  -\underset
{k=0}{\overset{h-1}{\sum}}s^{k\alpha}N^{k}B_{2}U(s) \label{infSFOSLap}%
\end{align}

where $X_{2}(s)=\mathcal{L}[x_{2}(t)],$ $U(s)=\mathcal{L}[u(t)].$ For the
piecewise continuously differentiable input $u(t)$, by invoking inverse
Laplace transform of (\ref{infSFOSLap}),we get%
\begin{align}
x_{2}(t)  &  =-\underset{k=1}{\overset{h-1}{%
{\textstyle \sum}
}}N^{k}\left(  \delta^{\left(  k\alpha-1\right)  }\left(  t\right)
x_{20}+\delta^{\left(  k\alpha-2\right)  }\left(  t\right)  x_{20}^{\left(
1\right)  }\right) \nonumber \\
&  -\underset{k=0}{\overset{h-1}{%
{\textstyle \sum}
}}N^{k}B_{2}\left(  D^{k\alpha}u\left(  t\right)  +\underset{j=0}%
{\overset{m-1}{%
{\textstyle \sum}
}}u^{(j)}(0)\delta^{\left(  k\alpha-j-1\right)  }\left(  t\right)  \right)
\label{infSFOSsolu}%
\end{align}

Now, we get state response (\ref{SFOS response}) with equations
(\ref{finSFOSsolu}) and (\ref{infSFOSsolu}).

For arbitrary initial conditions, some of them may not satisfy solution
(\ref{infSFOSsolu}) at $t=0$, which leads to discontinuous behavior at $t=0$.
Since discontinuous behavior is not desirable, the set of $x(0)$ which does
not result in discontinuous behavior at $t=0$ is called the set of
\textit{admissible initial conditions }\cite{Lewis1986survey}. The following
analyses the admissible initial conditions of system (\ref{RegSFOS}).

With $t\rightarrow0_{+}$, equation (\ref{infSFOSsolu}) turns into%
\begin{align*}
x_{2}(0_{+})=& -\underset{k=1}{\overset{h-1}{%
{\textstyle \sum}
}}N^{k}\left(  \delta^{\left(  k\alpha-1\right)  }\left(  0_{+}\right)
x_{20}+\delta^{\left(  k\alpha-2\right)  }\left(  0_{+}\right)  x_{20}%
^{\left(  1\right)  }\right)  \\
& -\underset{k=0}{\overset{h-1}{%
{\textstyle \sum}
}}N^{k}B_{2}\left(  D^{k\alpha}u\left(  0_{+}\right)  +\underset{j=0}%
{\overset{m-1}{%
{\textstyle \sum}
}}u^{(j)}(0)\delta^{\left(  k\alpha-j-1\right)  }\left(  0_{+}\right)
\right)
\end{align*}

i.e.%
\begin{align*}
& \left[  I+\underset{k=1}{\overset{h-1}{%
{\textstyle \sum}
}}N^{k}\delta^{\left(  k\alpha-1\right)  }\left(  0_{+}\right)  \right]
x_{20}\\
& =  -\underset{k=1}{\overset{h-1}{%
{\textstyle \sum}
}}N^{k}\delta^{\left(  k\alpha-2\right)  }\left(  0_{+}\right)  x_{20}%
^{\left(  1\right)  }\\
&\quad -\underset{k=0}{\overset{h-1}{%
{\textstyle \sum}
}}N^{k}B_{2}\left(  D^{k\alpha}u\left(  0_{+}\right)  +\underset{j=0}%
{\overset{m-1}{%
{\textstyle \sum}
}}u^{(j)}(0)\delta^{\left(  k\alpha-j-1\right)  }\left(  0_{+}\right)
\right)
\end{align*}

Let $M=\left[  I+\underset{k=1}{\overset{h-1}{%
{\textstyle \sum}
}}N^{k}\delta^{\left(k\alpha -1 \right)  }\left(
0_{+}\right)  \right]  $. $N^{k}$ is an upper triangular matrix and all the
elements of main diagonal are zero since $N$ is nilpotent. On the other hand,
if $(k\alpha-1)$ is not an integer, then $\delta^{\left(  k\alpha-1 \right)  }
\left(  0_{+}\right)$ can not be zero and it's a very large
number but not an infinite number. Therefore, the matrix $M$ is invertible and
the admissible initial conditions of system (\ref{RegSFOS}) is%
\begin{align*}
x_{20}= & -M^{-1}\Bigg(  \underset{k=1}{\overset{h-1}{
{\textstyle \sum}}}
N^{k}\delta^{(k\alpha-2)  }(  0_{+})  x_{20}^{(1)}\\
 & +\underset{k=0}{\overset{h-1}{
{\textstyle \sum}}}
N^{k}B_{2}\Big( D^{k\alpha}u( 0_{+}) +\underset{j=0}
{\overset{m-1}{{\textstyle \sum}}}u^{(j)}(0)\delta^{( k\alpha-j-1) }( 0_{+})
\Big)  \Bigg)
\end{align*}
The theorem is proved.
\end{proof}

\begin{definition}
For arbitrary initial conditions, if the state response to SFOS does not
include impulsive response, then the system is termed \textit{impulse-free}.
\end{definition}

Obiviously, the state response to SFOS (\ref{RegSFOS}) is similar to the state
response of SIOS. $x_{1}(t)$ is the state response to the finite subsystem,
which is represented by Mittag-Leffler function. $x_{2}(t)$ is the state
response to the infinite subsystem, which is composed of impulse function and
input function. Based on state response (\ref{SFOS response}), the following
analyses the impulsive behavior of system (\ref{RegSFOS}). Because substate
$x_{1}(t)$ is continuous, we focus on substate $x_{2}(t).$

(1) If $t=0$

Without loss of generality, let $u(t)=0$. If $x_{2}(0)\neq0$ and
$x_{2}(0)\notin \ker(N),$ there holds%

\[
x_{2}(t)=-\underset{k=1}{\overset{h-1}{%
{\textstyle \sum}
}}N^{k}\left(  \delta^{\left(  k\alpha-1\right)  }\left(  t\right)
x_{20}+\delta^{\left(  k\alpha-2\right)  }\left(  t\right)  x_{20}^{\left(
1\right)  }\right)
\]

If $t\rightarrow0,$ then $\delta^{\left(  \beta \right)  }\left(  t\right)
\rightarrow \infty$ $(\beta>0).$ Thus, $x_{20}$ which doesn't satisfy
admissible initial condition may reault in impulse.

(2) If $t>0$

$x_{2}(t)$ can be represented as
\begin{align}
x_{2}(t)= & -\underset{k=1}{\overset{h-1}{%
{\textstyle \sum}
}}N^{k}\left(  \delta^{\left(  k\alpha-1\right)  }\left(  t\right)
x_{20}+\delta^{\left(  k\alpha-2\right)  }\left(  t\right)  x_{20}^{\left(
1\right)  }\right) \notag\\
&  -\underset{k=0}{\overset{h-1}{%
{\textstyle \sum}
}}N^{k}B_{2}\left(  D^{k\alpha}u\left(  t\right)  +\underset{j=0}%
{\overset{m-1}{%
{\textstyle \sum}
}}u^{(j)}(0)\delta^{\left(  k\alpha-j-1\right)  }\left(  t\right)  \right)
\label{resX2}%
\end{align}
For $\delta^{\left(  b\right)  }$, if $b$ is a positive integer, the support
set of $\delta^{\left(  b\right)  }\left(  t\right)  $ is $\left \{  0\right \}
$, which means $\delta^{\left(  b\right)  }=0$; If $b$ is not a positive
integer, the support set of $\delta^{\left(  b\right)  }\left(  t\right)  $ is
not $\left \{  0\right \}  $, which means $\delta^{\left(  b\right)  }\left(
t\right)  \neq0$ when $t>0$. Therefore, if $N\neq0,$ then $x_{20}%
,x_{20}^{\left(  1\right)  },u^{(j)}(0)$ participate in the dynamic process of
$x_{2}(t)$ and the state response of SFOS includes impulsive response.

\begin{remark}
The $x_{20}$ and $u^{(j)}(0)$ of SFOS participate in the dynamic process of
substate $x_{2}(t)$, which is very different from SIOS.
\end{remark}

The final value theorem of fractional order system
\cite{Ortigueira2000Introduction},%

\begin{equation}
D^{\alpha-1}x(\infty)=\underset{s\rightarrow0}{\lim}s^{\alpha}X(s)\text{
\ }\operatorname{Re}(s)>0 \label{finalVT}%
\end{equation}

Therefore, we get $\delta^{(\beta)}\left(  t\right)  \rightarrow0$ when
$t\rightarrow \infty.$ It implies that terms including $\delta^{(\beta)}\left(
t\right)  $ on the right side of the equation (\ref{resX2}) do not impact on
the stability of the infinite subsystem (\ref{SFOSinfinite}), which is
convenient when analysing the stability of SFOS.

From time $0$ to time $t,$ the input $u(t)$ always has an influence on the
state $x_{2}(t)$ because of the properties of Caputo fractional derivative.
Thus, change of $u(t)$ can not be reflected immediately by substate $x_{2}(t)$
at the time $t$ and jump behavior will not appear in the state response.

\begin{remark}
The input $u(t)$ will not give rise to the jump behavior of $x_{2}(t)$, which
is also very different from SIOS.
\end{remark}

To sum up, we get the following theorem.

\begin{theorem}\label{theoImpsfe}
For arbitrary initial conditions, the regular SFOS (\ref{RegSFOS}) is
impulse-free if and only if $N=0$. $N$ comes from the decomposition
\begin{align*}
QEP=diag(I_{n_{1}},N)\\
QAP=diag(A_{1},I_{n_{2}})
\end{align*}
\end{theorem}

Similar to paper \cite{Ishihara2001Impulse}, the following gives another condition of impulse-free.

\begin{lemma}\label{lemSFOSIFb}
The following statements are equivalent:
\begin{enumerate}
  \item the regular system $(E,A,\alpha)$ is impulse-free;
  \item there exist a vector $\upsilon\in\mathbb{R}^n$ and a vector
  $\omega\in\mathbb{R}^n$ such that
  \begin{align*}
  E\upsilon &=0\\
  A\upsilon &=E\omega
  \end{align*}
  then $\upsilon=0$.
\end{enumerate}
\end{lemma}

\begin{proof}
According to Theorem \ref{theoSFOSregular}, regular system $(E,A,\alpha)$ has the decomposition that
\begin{align*}
QEP=diag(I,N)\\
QAP=diag(A,I)
\end{align*}

Thus,
\begin{align*}
& \left\{
\begin{array}{l}
E\upsilon=0\\
A\upsilon=E\omega
\end{array}
\right. \\
\Leftrightarrow &
\left\{
\begin{array}{l}
QEPP^{-1}\upsilon=0\\
QAPP^{-1}\upsilon=QEPP^{-1}\omega
\end{array}
\right. \\
\Leftrightarrow &
\left\{
\begin{array}{l}
\begin{bmatrix}
I & 0 \\
0 & N
\end{bmatrix}\begin{bmatrix}
\upsilon_{1}\\
\upsilon_{2}
\end{bmatrix}
=0\\
\begin{bmatrix}
A & 0 \\
0 & I
\end{bmatrix}\begin{bmatrix}
\upsilon_{1}\\
\upsilon_{2}
\end{bmatrix}=\begin{bmatrix}
I & 0 \\
0 & N
\end{bmatrix}\begin{bmatrix}
\omega_{1}\\
\omega_{2}
\end{bmatrix}
\end{array}
\right. \\
\Leftrightarrow &
\left\{
\begin{array}{l}
\upsilon_{1}=0\\
N\upsilon_{2}=0\\
A\upsilon_{1}=\omega_{1}\\
\upsilon_{2}=N\omega_{2}
\end{array}
\right. \\
\end{align*}
$\omega_{2}$ is not specific, therefore, $\upsilon_{2}=0$ if and only if $N=0$,
which is the sufficient and necessary condition of impulse-free.
Because $P$ is nonsingular, we can conclude that $\upsilon=0$ if and only if $N=0$.

This ends the proof.
\end{proof}

\begin{theorem}
\label{theoSFOSinfIF} The regular SFOS (\ref{RegSFOS}) is impulse-free if and only if there
exists no infinite eigenvector of order 2, i.e. $\upsilon^{2}$.
\end{theorem}

\begin{proof}
According to Lemma \ref{lemSFOSIFb}, the sufficient and necessary condition of  impulse-free is
\begin{align*}
E\upsilon^1 &=0\\
E\upsilon^2 &=A\upsilon^1=0
\end{align*}
 which implies that the infinite eigenvector of order 2 does not exist.

 This ends the proof.
\end{proof}

\section{Stability and Admissibility Analysis}

\subsection{Asymptotic Stability}

Stability is very important in control theory. In
\cite{NDoye2010Stabilization}, a sufficient condition of asymptotic stability
is given, but the condition demands free impulse. Meanwhile,
\cite{Yao2013Sufcient} also gives a sufficient condition of asymptotic
stability, but its fractional order is $0<\alpha<1$. The following gives a
sufficient and necessary condition of asymptotic stability with fractional
order $1<\alpha<2$, which is simpler than the condition in
\cite{NDoye2010Stabilization}. Consider the autonomous regular SFOS
\begin{equation}
ED^{\alpha}x\left(  t\right)  =Ax\left(  t\right)  \label{SFOS_stability}%
\end{equation}
where $x\left(  t\right)  \in%
\mathbb{R}
^{n}$, $1<\alpha<2$.

The following will analyse the asymptotic stability (stability for short) of
SFOS. Firstly, the definition of the stability of SFOS is given as follows.

\begin{definition}
For arbitrary admissible initial condition $x(0),$ if regular SFOS
(\ref{SFOS_stability}) satisfies $\underset{t\rightarrow+\infty}{\lim
}\left \Vert x(t)\right \Vert =0,$ then the SFOS (\ref{SFOS_stability}) is
called \textit{asymptotically stable}.
\end{definition}

The characteristic polynomial of system (\ref{SFOS_stability}) is
\begin{equation}
\Delta \left(  s\right)  =\det \left(  s^{\alpha}E-A\right)  =a_{n_{^{1}}%
}\left(  s^{\alpha}\right)  ^{n_{1}}+\cdots+a_{1}s^{\alpha}+a_{0}
\label{SFOS charactpoly}%
\end{equation}

It's obvious that the polynomial $\Delta \left(  s\right)  $ is a multivalued
function of $s$, of which the fractional degree is $n_{1}$ ($n_{1}\leq n$).
Let $s^{\alpha}=\omega,$ then $\Delta \left(  s\right)  $ turns into a
single-valued function $\overset{\_}{\Delta}(\omega)=\det(\omega E-A)$.
$\Delta \left(  s\right)  $ has a lot of roots but only the roots on the
principal Riemann surface $\Omega=\left \{  s\mid-\pi \leq \arg(s)<\pi \right \}  $
decide the time-domain behavior and stability of fractional system
(\cite{Gross1961Singularities,Beyer1995Definition}). Therefore, the physical
domain of $\Delta \left(  s\right)  $ is defined on the principal Riemann
surface. And the finite roots of $\Delta \left(  s\right)  $ on the principal
Riemann surface $\Omega$ are defined as the finite roots of SFOS
(\ref{SFOS_stability}).

\begin{lemma}
\label{lemStable} Fractional order system \cite{Moze2005LMI}
\begin{equation}
D^{\alpha}x\left(  t\right)  =Ax\left(  t\right)  ,\text{ \ }1<\alpha<2
\label{FOS model}%
\end{equation}
is asymptotically stable if and only if $\left \vert \arg \left(
spec(A)\right)  \right \vert $$>\alpha \pi/2$, where $spec(A)$ is the spectrum
(set of all eigenvalues) of $A$. Also, state vector $x(t)$ decays towards 0
and meets the following condition: $\left \Vert x(t)\right \Vert <Kt^{-\alpha
},t>0,K>0$.
\end{lemma}

\begin{theorem}
\label{theoStable} SFOS (\ref{SFOS_stability}) is asymptotically stable if and
only if%
\[
\left \vert \arg \left(  spec\left(  E,A\right)  \right)  \right \vert >\frac
{\pi}{2}\alpha
\]

where $spec(E,A)$ is the spectrum (set of all eigenvalues) of $(E,A,\alpha)$.
\end{theorem}

\begin{proof}
Because the system (\ref{SFOS_stability}) is regular, two nonsingular matrices
$Q,P$ may be chosen such that system (\ref{SFOS_stability}) is equivalent to
\begin{align}
D^{\alpha}x_{1}\left(  t\right)   &  =A_{1}x_{1}\left(  t\right)
\label{SFOSstab finite}\\
ND^{\alpha}x_{2}\left(  t\right)   &  =x_{2}\left(  t\right)
\label{SFOSstab infinite}%
\end{align}
where $\left[  x_{1}^{T}\left(  t\right)  \text{ \ }x_{2}^{T}\left(  t\right)
\right]  ^{T}=P^{-1}x$, $QEP=diag\left(  I_{n_{1}},N\right)  $,
$QAP=diag\left(  A_{1},I_{n-n_{1}}\right)  $. According to theorem
\ref{theoSFOSresponse}, the state response to system (\ref{SFOS_stability})
is
\begin{align}
x_{1}\left(  t\right)   &  =E_{\alpha,1}\left(  A_{1}t^{\alpha}\right)
x_{10}\nonumber \\
x_{2}\left(  t\right)   &  =-\underset{k=1}{\overset{h-1}{%
{\textstyle \sum}
}}N^{k}\left(  \delta^{\left(  k\alpha-1\right)  }\left(  t\right)
x_{20}+\delta^{\left(  k\alpha-2\right)  }\left(  t\right)  x_{20}^{\left(
1\right)  }\right)  \label{equx1x2}%
\end{align}
According to lemma \ref{lemStable}, finite subsystem (\ref{SFOSstab finite})
is stable if and only if 
\[
\left \vert \arg \left(  spec\left(  A_{1}\right)
\right)  \right \vert > \alpha \pi/2
\]

For the infinite subsystem (\ref{SFOSstab infinite}), according to the final
value theorem (\ref{finalVT}), when $t\rightarrow+\infty$, the state response
$x_{2}(t)\rightarrow0$. Thus, the infinite subsystem is essentially stable.

On the other hand, $\det(sN-I_{n-n_{1}})=(-1)^{n-n_{1}}$\ because $N$ is
nilpotent. Thus, $spec(N,I_{n-n_{1}})=\varnothing$ and
\begin{align*}
spec(E,A)  &  =spec(QEP,QAP)\\
&  =spec(diag(I_{n_{1}},N),diag(A_{1},I_{n-n_{1}}))\\
&  =spec(I_{n_{1}},A_{1})\cup spec(N,I_{n-n_{1}})\\
&  =spec(A_{1})\cup \varnothing \\
&  =spec(A_{1})
\end{align*}

i.e. $spec(E,A)=spec(A_{1})$. The theorem is proved.
\end{proof}

\subsection{Admissibility}

In \cite{Yao2013Sufcient}, a sufficient and necessary condition of
admissibility with fractional order $0<\alpha<1$ is given. The following gives
a sufficient and necessary condition of admissibility with fractional order
$1<\alpha<2$.

Similarly to the admissibility of SIOS, the following gives the definition of
admissibility for SFOS.

\begin{definition}
\label{defSFOSadimis} If a SFOS is regular, impulse-free and stable, then the
SFOS is termed \textit{admissible}.
\end{definition}

From the above analysis, we can know that the sufficient and necessary
conditions of regularity, free impulse, stability for SFOS are only related to
matrices $E,A$ and fractional order $\alpha$. Thus, the admissibility of SFOS
is only related to $E,A$ and $\alpha$.

According to theorem \ref{theoStable}, SFOS (\ref{SFOS_stability}) is stable
if and only if all the finite eigenvalues of SFOS belong to the region
$\Lambda=\left \{  \lambda \in%
\mathbb{C}
\mid \left \vert \arg(\lambda)\right \vert >\pi \alpha/2\right \}  .$ When
$1<\alpha<2,$ $\Lambda$ is a LMI region. Thus, we can analyse the
admissibility of system $(E,A,\alpha)$ via $\mathcal{D}$-stable theorem.

\begin{definition}
If system $(E,A,\alpha)$ is regular, impulse-free and all the finite
eigenvalues belong to the region $\mathcal{D}$, then system $(E,A,\alpha)$ is
termed $\mathcal{D}$-admissible.
\end{definition}

\begin{theorem}
\label{theoDStable} Let $rank(E)=r$, $E_{0}\in \mathbb{R}^{n\times \left(
n-r\right)  }$ column full rank and $E^{T}E_{0}=0$. SFOS is $\mathcal{D}%
$-admissible if and only if there exist symmetric positive matrix $P\in%
\mathbb{R}
^{n\times n}$ and matrix $Q\in%
\mathbb{R}
^{(n-r)\times n}$ such that%
\begin{equation}
M_{s}(E,A,P,Q)<0 \label{inequDStable}%
\end{equation}
where
\[
M_{s}\left(  E,A,P,Q\right)  =\Phi \otimes \left(  E^{T}PE\right)  +Sym\left \{
\Psi \otimes \left(  E^{T}PA\right)  +I_{d}\otimes \left(  Q^{T}E_{0}%
^{T}A\right)  \right \}
\]

$\Phi \in%
\mathbb{R}
^{d\times d}$ is a symmetric martrix and $\Psi \in%
\mathbb{R}
^{d\times d}$, $I_{d}$ has the dimension $d\times d$.
\end{theorem}

\begin{proof}

Sufficiency.

We will prove it by contradiction. Assume that SFOS is impulsive. According to
Theorem \ref{theoSFOSinfIF}, there exists eigenvector $\upsilon^{2}\in%
\mathbb{R}
^{n}$ such that $E\upsilon^{2}=A\upsilon^{1}$ and $E\upsilon^{1}=0.$ By left
multiplying $\left(  I_{d}\otimes v^{1}\right)  ^{T}$ and right multiplying
$\left(  I_{d}\otimes v^{1}\right)  $ on (\ref{inequDStable}), we have
\[
\left(  I_{d}\otimes v^{1}\right)  ^{T}M_{s}\left(  E,A,P,Q\right)  \left(
I_{d}\otimes v^{1}\right)  <0\
\]
i.e.
\begin{equation}
I_{d}\otimes \left[  \left(  v^{2}\right)  ^{T}\left(  E^{T}E_{0}Q+Q^{T}%
E_{0}^{T}E\right)  v^{2}\right]  <0\text{ } \label{inequSuff}%
\end{equation}
Because $E^{T}E_{0}=0$, the inequation (\ref{inequSuff}) can't be true. Thus,
system $(E,A,\alpha)$ is impulse-free.

Let $\lambda$ be the finite eigenvalue of system $(E,A,\alpha)$ and $\upsilon$
be the eigenvector, then we get $\ A\upsilon=\lambda E\upsilon$ and
$\upsilon^{\ast}A^{T}=\overset{\_}{\lambda}\upsilon^{\ast}E^{T}.$ From
inequality (\ref{inequDStable}), we get
\begin{align*}
&  \left(  I_{d}\otimes v\right)  ^{\ast}M_{s}\left(  E,A,P,Q\right)  \left(
I_{d}\otimes v\right) \\
=  &  \  \Phi \otimes \left(  v^{\ast}E^{T}PEv\right)  +\Psi \otimes \left(
\lambda v^{\ast}E^{T}PEv\right)  +\Psi^{T}\otimes \left(  \bar{\lambda}v^{\ast
}E^{T}PEv\right) \\
=  &  \ v^{\ast}E^{T}PEv\left(  \Phi+\lambda \Psi+\bar{\lambda}\Psi^{T}\right)
\\
=  &  \ v^{\ast}E^{T}PEvf_{\mathcal{D}}\left(  \lambda \right)  <0
\end{align*}
Because $P>0$, we can get $f_{\mathcal{D}}\left(  \lambda \right)  <0.$
According to the definition (\ref{defSFOSadimis}), we can conclude that system
$(E,A,\alpha)$ is $\mathcal{D}$-admissible.

Necessity.

Because the system $(E,A,\alpha)$ is regular and impulse-free, there exist two
nonsingular matrices $\ M$ and $N$ such that
\begin{align}
MEN  &  =\left[
\begin{array}
[c]{c}%
M_{1}\\
M_{2}%
\end{array}
\right]  E\left[
\begin{array}
[c]{cc}%
N_{1} & N_{2}%
\end{array}
\right]  =\left[
\begin{array}
[c]{cc}%
I_{r} & 0\\
0 & 0
\end{array}
\right]  \text{ }\nonumber \\
MAN  &  =\left[
\begin{array}
[c]{c}%
M_{1}\\
M_{2}%
\end{array}
\right]  A\left[
\begin{array}
[c]{cc}%
N_{1} & N_{2}%
\end{array}
\right]  =\left[
\begin{array}
[c]{cc}%
A_{1} & 0\\
0 & I_{n-r}%
\end{array}
\right]  \text{ } \label{equAdmNece}%
\end{align}
where $M_{1}\in%
\mathbb{R}
^{r\times n}$, $N_{1}\in%
\mathbb{R}
^{n\times r}$. Obiviously, $M_{2}AN=\left[
\begin{array}
[c]{cc}%
0 & I_{n-r}%
\end{array}
\right]  $, $M_{2}E=0$.

The system $(E,A,\alpha)$ is $\mathcal{D}$-admissible, thus we get
$\lambda(E,A)=\lambda(A_{1})$, i.e. $A_{1}$ is $\mathcal{D}$-stable. According
to lemma \ref{lemADstable}, there exists a symmetric real matrix $P_{1}>0$
such that
\[
\Phi \otimes P_{1}+\Psi \otimes \left(  P_{1}A_{1}\right)  +\Psi^{T}%
\otimes \left(  A_{1}^{T}P_{1}\right)  <0\text{ }%
\]
Since it's a strict inequality, there must exist a sufficiently small
$\varepsilon>0$ such that
\[
\Phi \otimes P_{1}+\Psi \otimes \left(  P_{1}A_{1}\right)  +\Psi^{T}%
\otimes \left(  A_{1}^{T}P_{1}\right)  +I_{d}\otimes \left(  \frac{\varepsilon
}{2}N_{1}^{T}N_{1}\right)  <0\text{ }%
\]
i.e.
\begin{gather}
\Phi \otimes P_{1}+\Psi \otimes \left(  P_{1}A_{1}\right)  +\Psi^{T}%
\otimes \left(  A_{1}^{T}P_{1}\right)  +\left[  I_{d}\otimes \left(  \varepsilon
N_{1}^{T}N_{2}\right)  \right] \nonumber \\
\left[  I_{d}\otimes \left(  2\varepsilon N_{2}^{T}N_{2}\right)  \right]
^{-1}\left[  I_{d}\otimes \left(  \varepsilon N_{2}^{T}N_{1}\right)  \right]
<0 \label{inequAdmiNece}%
\end{gather}
Invoking Schur complement, inequality (\ref{inequAdmiNece}) is equivalent to
\begin{gather*}
\left[
\begin{array}
[c]{cc}%
\Phi \otimes P_{1}+\Psi \otimes \left(  P_{1}A_{1}\right)  +\Psi^{T}%
\otimes \left(  A_{1}^{T}P_{1}\right)  & -I_{d}\otimes \left(  \varepsilon
N_{1}^{T}N_{2}\right) \\
\bullet & -I_{d}\otimes \left(  2\varepsilon N_{2}^{T}N_{2}\right)
\end{array}
\right]  <0\\
\Leftrightarrow \Phi \otimes \left[
\begin{array}
[c]{cc}%
P_{1} & 0\\
0 & 0
\end{array}
\right]  +Sym\left \{  \Psi \otimes \left[
\begin{array}
[c]{cc}%
P_{1}A_{1} & 0\\
0 & 0
\end{array}
\right]  +I_{d}\otimes \left(  \left[
\begin{array}
[c]{c}%
0\\
I_{n-r}%
\end{array}
\right]  \right.  \right. \\
\left.  \left.  \left[
\begin{array}
[c]{cc}%
-\varepsilon N_{2}^{T}N_{1} & -\varepsilon N_{2}^{T}N_{2}%
\end{array}
\right]  \right)  \right \}  <0\\
\Leftrightarrow \Phi \otimes \left(  \left[
\begin{array}
[c]{cc}%
I_{r} & 0\\
0 & 0
\end{array}
\right]  \left[
\begin{array}
[c]{cc}%
P_{1} & 0\\
0 & I_{n-r}%
\end{array}
\right]  \left[
\begin{array}
[c]{cc}%
I_{r} & 0\\
0 & 0
\end{array}
\right]  \right)  +Sym\left \{  \Psi \otimes \left(  \left[
\begin{array}
[c]{cc}%
I_{r} & 0\\
0 & 0
\end{array}
\right]  \right.  \right. \\
\left.  \left.  \left[
\begin{array}
[c]{cc}%
P_{1} & 0\\
0 & I_{n-r}%
\end{array}
\right]  \left[
\begin{array}
[c]{cc}%
A_{1} & 0\\
0 & I_{n-r}%
\end{array}
\right]  \right)  +I_{d}\otimes \left(  \left[
\begin{array}
[c]{c}%
0\\
I_{n-r}%
\end{array}
\right]  \left(  -\varepsilon N_{2}^{T}\right)  N\right)  \right \}  <0
\end{gather*}
Utilizing (\ref{equAdmNece}) and let $\hat{P}=\left[
\begin{array}
[c]{cc}%
P_{1} & 0\\
0 & I_{n-r}%
\end{array}
\right]  >0$, we have
\begin{gather*}
\Phi \otimes \left(  N^{T}E^{T}M^{T}\hat{P}MEN\right)  +Sym\left \{  \Psi
\otimes \left(  N^{T}E^{T}M^{T}\hat{P}MAN\right)  \right. \\
\left.  +I_{d}\otimes \left(  N^{T}A^{T}M_{2}^{T}\left(  -\varepsilon N_{2}%
^{T}\right)  N\right)  \right \}  <0
\end{gather*}
Let $M^{T}\hat{P}M=P$, $M_{2}^{T}=E_{0}$ and $-\varepsilon N_{2}^{T}=Q$, where
$E_{0}$ is column full rank and $E^{T}E_{0}=0$. Since $N$ is nonsingular, we
get
\[
\Phi \otimes \left(  E^{T}PE\right)  +\Psi \otimes \left(  E^{T}PA\right)
+\Psi^{T}\otimes \left(  A^{T}PE\right)  +I_{d}\otimes \left(  A^{T}E_{0}%
Q+Q^{T}E_{0}^{T}A\right)  <0
\]
The theorem is proved.
\end{proof}

The condition (\ref{inequDStable}) is a strict linear matrix inequality. In
order to analyse the Robust problems of SFOS conveniently, the nonstrict LMI
condition is given as follows.

\begin{lemma}
\label{lemDgamaAdmis}\cite{Hsiung1997Pole} SIOS $(E_{I},A_{I})$ is
$\mathcal{D}_{\Gamma}$ (when$\  \Phi=0$) admissible if and only if there exists
a matrix $P\in \mathbb{R}^{n\times n}$ such that
\begin{align}
Sym\left \{  \Psi \otimes \left(  PA_{I}\right)  \right \}   &
<0\label{inequDgamaAdmis}\\
PE_{I} =E_{I}^{T}P^{T}  &  \geq0
\end{align}

\end{lemma}

\begin{remark}
Note that conditions (\ref{inequDStable}) and (\ref{inequDgamaAdmis}) do not
have to be regular and impulse-free, thus they can be used generally. For
$(E,A,\alpha)$, replace $E_{I},A_{I}$ by $E,A$\ respectively, the lemma
\ref{lemDgamaAdmis} is also true because the eigenvalues of SFOS and SIOS are equivalent.
\end{remark}

\begin{lemma}
\label{lemAdmIneqRes} $(E,A,\alpha)$ with $1<\alpha<2$ is admissible if and
only if there exist matrices $P=P^{T}>0$, $P\in%
\mathbb{R}
^{n\times n}$, $Q=Q^{T}$, $Q\in%
\mathbb{R}
^{n\times n}$ such that
\begin{align}
Sym\left \{  \Theta \otimes \left(  A^{T}PE\right)  \right \}  +I_{2}%
\otimes \left(  A^{T}QA\right)   &  <0\label{ineAdmine1}\\
E^{T}QE  &  \geq0 \label{ineAdmine2}%
\end{align}
where $\Theta=\left[
\begin{array}
[c]{cc}%
\sin \frac{\pi}{2}\alpha & -\cos \frac{\pi}{2}\alpha \\
\cos \frac{\pi}{2}\alpha & \sin \frac{\pi}{2}\alpha
\end{array}
\right]  $, $I_{2}$ has the same dimension with $\Theta$.
\end{lemma}

\begin{proof}
Sufficiency. Assume that $(E,A,\alpha)$ is impulsive, then there exists an
infinite eigenvector of order 2, $\upsilon^{2}\in%
\mathbb{R}
^{n}$ such that $E\upsilon^{2}=A\upsilon^{1}$, $E\upsilon^{1}=0$. By left
multiplying $(I_{2}\otimes \upsilon^{1})^{\ast}$ and right multiplying
$(I_{2}\otimes \upsilon^{1})$ on (\ref{ineAdmine1}), we have%
\begin{align}
&  (I_{2}\otimes \upsilon^{1})^{\ast}\left[  Sym\left \{  \Theta \otimes
(A^{T}PE)\right \}  +I_{2}\otimes(A^{T}QA)\right]  (I_{2}\otimes \upsilon
^{1})\nonumber \\
=  &  I_{2}\otimes \left[  (\upsilon^{1})^{\ast}A^{T}QA\upsilon^{1}\right]
\nonumber \\
=  &  I_{2}\otimes \left[  (\upsilon^{2})^{\ast}E^{T}QE\upsilon^{2}\right]  <0
\label{ineAdmine3}%
\end{align}
Inequality (\ref{ineAdmine3}) can't be true because $E^{T}QE\geq0$. Thus,
system $(E,A,\alpha)$ is impulse-free and two matrices $M$ and $N$ may be
chosen such that
\begin{equation}
MEN=%
\begin{bmatrix}
I_{m} & 0\\
0 & 0
\end{bmatrix}
,MAN=%
\begin{bmatrix}
A_{1} & 0\\
0 & I_{n-m}%
\end{bmatrix}
\end{equation}
Let $Y=M^{-T}PM^{-1}=%
\begin{bmatrix}
Y_{11} & Y_{22}\\
Y_{12}^{\ast} & Y_{22}%
\end{bmatrix}
$. Obviously, there holds $Y=Y^{\ast}>0$. Let $\widehat{Q}=M^{-T}QM^{-1}=%
\begin{bmatrix}
\widehat{Q}_{11} & \widehat{Q}_{22}\\
\widehat{Q}_{12}^{T} & \widehat{Q}_{22}%
\end{bmatrix}
$. From (\ref{ineAdmine2}) we have%
\begin{align}
&  N^{T}E^{T}QEN=N^{T}E^{T}M^{T}\widehat{Q}MEN\nonumber \\
=  &
\begin{bmatrix}
I_{m} & 0\\
0 & 0
\end{bmatrix}%
\begin{bmatrix}
\widehat{Q}_{11} & \widehat{Q}_{12}\\
\widehat{Q}_{12}^{T} & \widehat{Q}_{22}%
\end{bmatrix}%
\begin{bmatrix}
I_{m} & 0\\
0 & 0
\end{bmatrix}
\nonumber \\
=  &
\begin{bmatrix}
\widehat{Q}_{11} & 0\\
0 & 0
\end{bmatrix}
\geq0
\end{align}
Thus, we get $\widehat{Q}_{11}\geq0$.

Left multiply $(I_{2}\otimes N)^{T}$ and right multiply $I_{2}\otimes N$ on
(\ref{ineAdmine1}), we get
\begin{align}
&  (I_{2}\otimes N)^{T}\left(  Sym\left \{  \Theta \otimes(A^{T}PE)\right \}
+I_{2}\otimes(A^{T}QA)\right)  (I_{2}\otimes N)\nonumber \\
=  &  Sym\left \{  \Theta \otimes(N^{T}A^{T}M^{T}PMEN)\right \}  +I_{2}%
\otimes(N^{T}A^{T}M^{T}\widehat{Q}MAN)\nonumber \\
=  &  Sym\left \{  \Theta \otimes \left(
\begin{bmatrix}
A_{1}^{T} & 0\\
0 & I_{n-m}%
\end{bmatrix}%
\begin{bmatrix}
Y_{11} & Y_{12}\\
\bullet & Y_{22}%
\end{bmatrix}%
\begin{bmatrix}
I_{m} & 0\\
0 & 0
\end{bmatrix}
\right)  \right \} \nonumber \\
&  +I_{2}\otimes \left(
\begin{bmatrix}
A_{1}^{T} & 0\\
0 & I_{n-m}%
\end{bmatrix}%
\begin{bmatrix}
\widehat{Q}_{11} & \widehat{Q}_{12}\\
\bullet & \widehat{Q}_{22}%
\end{bmatrix}%
\begin{bmatrix}
A_{1} & 0\\
0 & I_{n-m}%
\end{bmatrix}
\right) \nonumber \\
=  &  Sym\left \{  \Theta \otimes%
\begin{bmatrix}
A_{1}^{T}Y_{11} & 0\\
Y_{12}^{T} & 0
\end{bmatrix}
\right \}  +I_{2}\otimes%
\begin{bmatrix}
A_{1}^{T}\widehat{Q}_{11}A_{1} & A_{1}^{T}\widehat{Q}_{12}\\
\bullet & \widehat{Q}_{22}%
\end{bmatrix}
<0 \label{ineAdmine4}%
\end{align}
According to \cite{Brewer1978Kronecker}, inequality (\ref{ineAdmine4}) is
equivalent to
\begin{equation}
Sym\left \{
\begin{bmatrix}
\Theta \otimes(A_{1}^{T}Y_{11}) & 0\\
\Theta \otimes Y_{12}^{T} & 0
\end{bmatrix}
\right \}  +%
\begin{bmatrix}
I_{2}\otimes(A_{1}^{T}\widehat{S}_{11}A_{1}) & I_{2}\otimes(A_{1}^{T}%
\widehat{Q}_{12})\\
\bullet & I_{2}\otimes \widehat{Q}_{22}%
\end{bmatrix}
<0
\end{equation}
Thus, we get
\begin{equation}
Sym\left \{  \Theta \otimes(A_{1}^{T}Y_{11})\right \}  +I_{2}\otimes \left(
A_{1}^{T}\widehat{Q}_{11}A_{1}\right)  <0
\end{equation}

Because $\widehat{Q}_{11}\geq0$, we have $Sym\left \{  \Theta \otimes(A_{1}%
^{T}Y_{11})\right \}  <0$. According to lemma \ref{lemNorStable}, $A_{1}$ is
stable. Thus $(E,A,\alpha)$ is stable. Finally, the admissibility of
$(E,A,\alpha)$ is achieved.

Necessary.

$(E,A,\alpha)$ is admissible, thus there exist two nonsingular matrices
$M,N,Y_{11}$ such that
\begin{align*}
&  MEN=%
\begin{bmatrix}
I_{m} & 0\\
0 & 0
\end{bmatrix}
,MAN=%
\begin{bmatrix}
A_{1} & 0\\
0 & I_{n-m}%
\end{bmatrix}
,\\
&  Sym\left \{  \Theta \otimes(A_{1}^{T}Y_{11})\right \}  <0.
\end{align*}

For a sufficiently small $\varepsilon >0$, we have%
\begin{equation}
Sym\left \{  \Theta \otimes(A_{1}^{T}Y_{11})\right \}  +I_{2}\otimes \left(
A_{1}^{T}\varepsilon Y_{11}A_{1}\right)  <0 \label{ineAdmine5}%
\end{equation}

Note 
\begin{align*}
P & =M^{T}YM=M^{T}%
\begin{bmatrix}
Y_{11} & 0\\
0 & I_{n-m}%
\end{bmatrix}M \\
Q & =M^{T}\widehat{Q}M=M^{T}%
\begin{bmatrix}
\widehat{Q}_{11} & 0\\
0 & \widehat{Q}_{2}%
\end{bmatrix}M
\end{align*}
where $\widehat{Q}_{11}=\varepsilon Y_{11}$ and $\widehat{Q}_{22}$ be any
negative definite matrix. From (\ref{ineAdmine5}) we get
\begin{equation}%
\begin{bmatrix}
Sym\left \{  \Theta \otimes(A_{1}^{T}Y_{11})\right \}  & 0\\
0 & 0
\end{bmatrix}
+%
\begin{bmatrix}
I_{2}\otimes(A_{1}^{T}\varepsilon Y_{11}A_{1}) & 0\\
0 & I_{2}\otimes \widehat{Q}_{22}%
\end{bmatrix}
<0.
\end{equation}

which is equivalent to%
\begin{align*}
&  Sym\left \{  \Theta \otimes%
\begin{bmatrix}
A_{1}^{T}Y_{11} & 0\\
0 & 0
\end{bmatrix}
\right \}  +I_{2}\otimes%
\begin{bmatrix}
A_{1}^{T}\varepsilon Y_{11}A_{1} & 0\\
0 & \widehat{Q}_{22}%
\end{bmatrix}
<0\\
\Leftrightarrow &  Sym\left \{  \Theta \otimes \left(  N^{T}A^{T}M^{T}%
PMEN\right)  \right \}  +I_{2}\otimes \left(  N^{T}A^{T}M^{T}QMAN\right)  <0\\
\Leftrightarrow &  (I_{2}\otimes N^{T})\left[  Sym\left \{  \Theta \otimes
(A^{T}PE)\right \}  +I_{2}\otimes(A^{T}QA)\right]  (I_{2}\otimes N)<0\\
\Leftrightarrow &  Sym\left \{  \Theta \otimes(A^{T}PE)\right \}  +I_{2}%
\otimes(A^{T}QA)<0
\end{align*}

and%
\begin{align*}
&
\begin{bmatrix}
\widehat{Q}_{11} & 0\\
0 & 0
\end{bmatrix}
\geq0\\
&  \Leftrightarrow N^{T}E^{T}M^{T}\widehat{Q}MEN\geq0\\
&  \Leftrightarrow N^{T}E^{T}QEN\geq0\\
&  \Leftrightarrow E^{T}QE\geq0
\end{align*}

This completes the proof.
\end{proof}

\begin{theorem}
\label{theoADM} The following statements are equivalent

\begin{enumerate}
\item System $(E,A,\alpha)$ ($1<\alpha<2$) is admissible;

\item Assume that $E_{0}\in \mathbb{R}^{n\times(n-r)}$ is column full rank and
$E^{T}E_{0}=0$, there exist symmetric positive matrix $P\in%
\mathbb{R}
^{n\times n}$ and metrix $Q\in \mathbb{R}^{(n-r)\times n}$ such that%
\begin{equation}
Sym\left \{  \Theta \otimes \left(  E^{T}PA\right)  +I\otimes \left(  Q^{T}%
E_{0}^{T}A\right)  \right \}  <0\label{inequLMIun}%
\end{equation}

\item There exists matrix $P\in \mathbb{R}^{n\times n}$ such that
\begin{gather}
Sym\left \{  \Theta \otimes \left(  PA\right)  \right \}  <0\nonumber \\
PE=E^{T}P^{T}\geq0
\end{gather}

\item there exist symmetric positive matrix $P\in%
\mathbb{R}
^{n\times n}$ and symmetric matrix $Q\in%
\mathbb{R}
^{n\times n}$ such that
\begin{align}
Sym\left \{  \Theta \otimes \left(  E^{T}PA\right)  \right \}  +I\otimes \left(
A^{T}QA\right)   &  <0\\
E^{T}QE  &  \geq0\nonumber
\end{align}

where $\Theta=\left[
\begin{array}
[c]{cc}%
\sin \frac{\pi}{2}\alpha & -\cos \frac{\pi}{2}\alpha \\
\cos \frac{\pi}{2}\alpha & \sin \frac{\pi}{2}\alpha
\end{array}
\right]  $ and $I$ has the same dimension with $\Theta$.
\end{enumerate}
\end{theorem}

\begin{proof}
When $1<\alpha<2$, the stable region of SFOS is a LMI region, in which case
$\Phi=0$ and $\Psi=\left[
\begin{array}
[c]{cc}%
\sin \frac{\pi}{2}\alpha & -\cos \frac{\pi}{2}\alpha \\
\cos \frac{\pi}{2}\alpha & \sin \frac{\pi}{2}\alpha
\end{array}
\right]  $. For such $\Phi$ and $\Psi$ and according to Theorem \ref{theoDStable}, Lemma \ref{lemDgamaAdmis} and Lemma \ref{lemAdmIneqRes}%
, the conclusion is achieved.
\end{proof}

\section{Robust admissibility analysis}

To the best of our knowledge, there exists no research on robust admissibility
of uncertain SFOS. In this section, sufficient conditions are given to check
the robust admissibility of the uncertain SFOS.

Consider the following uncertain SFOS%

\begin{equation}
ED^{\alpha}x(t)=Ax(t)=(A_{0}+D_{A}F_{A}E_{A})x(t) \label{sysUSFOS}%
\end{equation}

where $1<\alpha<2$ and $A_{0}\in \mathbb{R}^{n\times n}$, $D_{A}\in
\mathbb{R}^{n\times p},E_{A}\in \mathbb{R}^{q\times n}$ are given certain
matrices. The uncertain matrix $F_{A}\in \mathbb{R}^{p\times q}$ satisfies
\begin{equation}
F_{A}F_{A}^{T}<I_{p} \label{ineUSFOS}%
\end{equation}

\begin{theorem}
System (\ref{sysUSFOS}) is robust admissible if there exist matrices
$X=X^{T}>0$, $X\in%
\mathbb{R}
^{n\times n}$, $S\in%
\mathbb{R}
^{(n-m)\times n}$ such that%
\begin{equation}
\begin{bmatrix}
Z_{11} & Z_{12} & Z_{13}\\
\bullet & Z_{22}  & Z_{23}\\
\bullet & \bullet & Z_{33}
\end{bmatrix}  <0 \label{ineRobAdm1}
\end{equation}

where $E^{T}E_{0}=0$, $I_{2}$ is a $2\times2$ matrix and $\Theta=\left[
\begin{array}
[c]{cc}%
\sin \frac{\pi}{2}\alpha & -\cos \frac{\pi}{2}\alpha \\
\cos \frac{\pi}{2}\alpha & \sin \frac{\pi}{2}\alpha
\end{array}
\right]  $ and

\begin{align*}
Z_{11}= & Sym\left \{  \Theta \otimes(A_{0}^{T}XE)+I_{2}\otimes(A_{0}^{T}E_{0}S)\right \}
+2I_{2}\otimes(E_{A}^{T}E_{A})\\
Z_{12}= & I_{2}\otimes(E^{T}XD_{A})\\
Z_{13}= & I_{2}\otimes(S^{T}E_{0}^{T}D_{A})\\
Z_{22}= & -I_{2}\otimes I \\
Z_{23}= & 0\\
Z_{33}= & -I_{2}\otimes I
\end{align*}
\end{theorem}

\begin{proof}
Invoking Schur complement, inequality (\ref{ineRobAdm1}) is equivalent to%
\begin{align}
&  Sym\left \{  \Theta \otimes(A_{0}^{T}XE)+I_{2}\otimes(A_{0}^{T}%
E_{0}S)\right \}  +2I_{2}\otimes(E_{A}^{T}E_{A})\nonumber \\
+  &  I_{2}\otimes \left(  (E^{T}XD_{A}D_{A}^{T}XE)\right)  +I_{2}%
\otimes \left(  (S^{T}E_{0}^{T}D_{A}D_{A}^{T}E_{0}S)\right)  <0
\label{ineRobAdm2}%
\end{align}

From (\ref{ineUSFOS}) and (\ref{ineRobAdm1}), we get%
\begin{align}
&  Sym\left \{  \Theta \otimes(A_{0}^{T}XE)+I_{2}\otimes(A_{0}^{T}%
E_{0}S)\right \}  +(\Theta \Theta^{T})\otimes(E_{A}^{T}E_{A})+I_{2}\otimes
(E_{A}^{T}E_{A})\nonumber \\
+  &  I_{2}\otimes \left(  (E^{T}XD_{A}F_{A}F_{A}^{T}D_{A}^{T}XE)\right)
+I_{2}\otimes \left(  (S^{T}E_{0}^{T}D_{A}F_{A}F_{A}^{T}D_{A}^{T}%
E_{0}S)\right)  <0 \label{ineRobAdm3}%
\end{align}

According to lemma \ref{lemIneMa} and inequality (\ref{ineRobAdm3}), we get%
\begin{align*}
&  Sym\left \{  \Theta \otimes(A_{0}^{T}XE)+I_{2}\otimes(A_{0}^{T}%
E_{0}S)\right \} \\
&  +Sym\left \{  \Theta \otimes((D_{A}F_{A}E_{A})^{T}XE)\right \}  +Sym\left \{
I_{2}\otimes((D_{A}F_{A}E_{A})^{T}(E_{0}S))\right \}  <0\\
\Leftrightarrow &  Sym\left \{  \Theta \otimes \left(  (A_{0}+D_{A}F_{A}%
E_{A})^{T}XE\right)  +I_{2}\otimes \left(  (A_{0}+D_{A}F_{A}E_{A})^{T}%
E_{0}S\right)  \right \}  <0\\
\Leftrightarrow &  Sym\left \{  \Theta \otimes \left(  A^{T}XE\right)
+I_{2}\otimes \left(  A^{T}E_{0}S\right)  \right \}  <0
\end{align*}

Thus, according to theorem \ref{theoADM}, system (\ref{sysUSFOS}) is robust admissible.

The theorem is proved.
\end{proof}

\begin{theorem}
System (\ref{sysUSFOS}) is robust admissible if there exist matrices
$X=X^{T}>0,X\in \mathbb{R}^{n\times n},Y=Y^{T}>0,Y\in \mathbb{R}^{n\times n}$
and $S=S^{T},S\in \mathbb{R}^{n\times n}$ such that%
\begin{equation}
\begin{bmatrix}
Z_{11} & Z_{12} & Z_{13}\\
\bullet & Z_{22} & Z_{23}\\
\bullet & \bullet & Z_{33}
\end{bmatrix}  <0 \label{ineRoAdmine1}%
\end{equation}%
\begin{equation}
E^{T}SE\geq0 \label{ineRoAdmine2}%
\end{equation}

where $\Theta=\left[
\begin{array}
[c]{cc}%
\sin \frac{\pi}{2}\alpha & -\cos \frac{\pi}{2}\alpha \\
\cos \frac{\pi}{2}\alpha & \sin \frac{\pi}{2}\alpha
\end{array}
\right]  $ and 
\begin{align*}
Z_{11}= & Sym\left \{  \Theta \otimes(A_{0}^{T}XE)\right \}  
+I_{2}\otimes(A_{0}^{T}SA_{0})\\
Z_{12}= & \Theta^{T}\otimes(E^{T}XD_{A})+I_{2}\otimes(A_{0}^{T}SD_{A})\\
Z_{13}= & I_{2}\otimes YE_{A}^{T}\\
Z_{22}= & I_{2}\otimes(S-Y)\\
Z_{23}= & 0\\
Z_{33}= & -I_{2}\otimes Y
\end{align*}
\end{theorem}

\begin{proof}
According to theorem \ref{theoADM}, uncertain system (\ref{sysUSFOS}) is
robust admissible if for any $F_{A}$, there holds
\begin{align*}
&  Sym\left \{  \Theta \otimes \left[  (A_{0}+D_{A}F_{A}E_{A})^{T}XE\right]
\right \} \\
+  &  I_{2}\otimes \left[  (A_{0}+D_{A}F_{A}E_{A})^{T}S(A_{0}+D_{A}F_{A}%
E_{A})\right]  <0
\end{align*}

i.e.%
\begin{align}
&  \left(  I_{2}\otimes \upsilon^{T}\right)  (Sym\left \{  \Theta \otimes \left[
(A_{0}+D_{A}F_{A}E_{A})^{T}XE\right]  \right \} \nonumber \\
+  &  I_{2}\otimes \left[  (A_{0}+D_{A}F_{A}E_{A})^{T}S(A_{0}+D_{A}F_{A}%
E_{A})\right]  )\left(  I_{2}\otimes \upsilon \right)  <0 \label{ineRoAdmine3}%
\end{align}

The inequality (\ref{ineRoAdmine3}) can be rewritten as%
\begin{align}
&
\begin{bmatrix}
I_{2}\otimes \upsilon^{T}I_{2}\otimes(F_{A}E_{A}\upsilon)^{T}%
\end{bmatrix}
\nonumber \\
\times &
\begin{bmatrix}
Sym\left \{  \Theta \otimes(A_{0}^{T}XE)\right \}  +I_{2}\otimes(A_{0}^{T}%
SA_{0}) & \Theta^{T}\otimes(E^{T}XD_{A})+I_{2}\otimes(A_{0}^{T}SD_{A})\\
\bullet & I_{2}\otimes S
\end{bmatrix}
\nonumber \\
\times &
\begin{bmatrix}
I_{2}\otimes \upsilon \\
I_{2}\otimes(F_{A}E_{A}\upsilon)
\end{bmatrix}
<0 \label{ineRoAdmine4}%
\end{align}

From inequality (\ref{ineUSFOS}), we get%
\begin{align}
&  \left[  I_{2}\otimes \upsilon^{T}I_{2}\otimes(F_{A}E_{A}\upsilon
)^{T}\right]
\begin{bmatrix}
I_{2}\otimes(E_{A}^{T}E_{A}) & 0\\
0 & -I_{2}\otimes I
\end{bmatrix}%
\begin{bmatrix}
I_{2}\otimes \upsilon \\
I_{2}\otimes(F_{A}E_{A}\upsilon)
\end{bmatrix}
\nonumber \\
=  &  \left(  I_{2}\otimes \upsilon^{T}\right)  \left[  I_{2}\otimes \left(
E_{A}^{T}E_{A}-E_{A}^{T}F_{A}^{T}F_{A}E_{A}\right)  \right]  \left(
I_{2}\otimes \upsilon \right) \nonumber \\
=  &  \left(  I_{2}\otimes \upsilon^{T}\right)  \left[  I_{2}\otimes \left(
E_{A}^{T}E_{A}(I-F_{A}^{T}F_{A})\right)  \right]  \left(  I_{2}\otimes
\upsilon \right)  >0 \label{ineRoAdmine5}%
\end{align}

By applying the $S$-procedure, inequalities (\ref{ineRoAdmine4}) and
(\ref{ineRoAdmine5}) derive that there exists some scalar $\tau>0$ such that%
\begin{align}
&
\begin{bmatrix}
Sym\left \{  \Theta \otimes(A_{0}^{T}XE)\right \}  +I_{2}\otimes(A_{0}^{T}%
SA_{0}) & \Theta^{T}\otimes(E^{T}XD_{A})+I_{2}\otimes(A_{0}^{T}SD_{A})\\
\bullet & I_{2}\otimes S
\end{bmatrix}
\nonumber \\
&  +\tau%
\begin{bmatrix}
I_{2}\otimes(E_{A}^{T}E_{A}) & 0\\
0 & -I_{2}\otimes I
\end{bmatrix}
<0\nonumber \\
\Leftrightarrow &
\begin{bmatrix}
Sym\left \{  \Theta \otimes(A_{0}^{T}XE)\right \}  +I_{2}\otimes(A_{0}^{T}%
SA_{0}) & \Theta^{T}\otimes(E^{T}XD)+I_{2}\otimes(A_{0}^{T}SD)\\
\bullet & I_{2}\otimes S-I_{2}\otimes(\tau I)
\end{bmatrix}
\nonumber \\
&  +\tau%
\begin{bmatrix}
I_{2}\otimes E_{A}^{T}\\
0
\end{bmatrix}%
\begin{bmatrix}
I_{2}\otimes E_{A} & 0
\end{bmatrix}
<0 \label{ineRoAdmine6}%
\end{align}
Let $Y=\tau I$ and invoking Schur Complement, inequality (\ref{ineRoAdmine1})
is obtained.

The theorem is proved.
\end{proof}

\begin{theorem}
System (\ref{sysUSFOS}) is robust admissible if there exist matrix $X\in%
\mathbb{R}
^{n\times n}$, such that%
\begin{equation}
\left[
\begin{array}
[c]{cc}%
Sym\left \{  \Theta \otimes(A_{0}^{T}X)\right \}  +I_{2}\otimes(E_{A}^{T}E_{A}) &
I_{2}\otimes(XD_{A})\\
\bullet & -I_{2}\otimes I
\end{array}
\right]  <0 \label{ineRoeq1}%
\end{equation}%
\begin{equation}
E^{T}X=XE\geq0 \label{ineRoeq2}%
\end{equation}

where $\Theta=\left[
\begin{array}
[c]{cc}%
\sin \frac{\pi}{2}\alpha & -\cos \frac{\pi}{2}\alpha \\
\cos \frac{\pi}{2}\alpha & \sin \frac{\pi}{2}\alpha
\end{array}
\right]  $.
\end{theorem}

\begin{proof}
Invoking Schur Complement, inequality (\ref{ineRoeq1}) is equivalent to
\begin{equation}
Sym\left \{  \Theta \otimes(A_{0}^{T}X)\right \}  +I_{2}\otimes(E_{A}^{T}%
E_{A})+I_{2}\otimes(XD_{A}D_{A}^{T}X)<0 \label{ineRoeq3}%
\end{equation}

From inequalities (\ref{ineUSFOS}) and (\ref{ineRoeq3}), we get%
\begin{equation}
Sym\left \{  \Theta \otimes(A_{0}^{T}X)\right \}  +(\Theta \Theta^{T}%
)\otimes(E_{A}^{T}E_{A})+I_{2}\otimes(XD_{A}F_{A}F_{A}^{T}D_{A}^{T}X)<0
\label{ineRoeq4}%
\end{equation}

According to lemma \ref{lemIneMa} and inequality (\ref{ineRoeq4}), we have%
\begin{align*}
&  Sym\left \{  \Theta \otimes(A_{0}^{T}X)\right \}  +Sym\left \{  \Theta
\otimes((D_{A}F_{A}E_{A})^{T}X)\right \}  <0\\
\Leftrightarrow &  Sym\left \{  \Theta \otimes \left(  (A_{0}+D_{A}F_{A}%
E_{A})^{T}X\right)  \right \}  <0\\
\Leftrightarrow &  Sym\left \{  \Theta \otimes \left(  A^{T}X\right)  \right \}
<0.
\end{align*}

Thus, according to theorem \ref{theoADM}, the system (\ref{sysUSFOS}) is
robust admissible.

The theorem is proved.
\end{proof}

\section{Numerical examples}

\subsection{Numerical solution in time domain}

Now we will get the numerical solution of SFOS.

System $(E,A,\alpha)$ can be decomposed into%
\[
\left \{
\begin{array}
[c]{c}%
D^{\alpha}x_{1}\left(  t\right)  =A_{1}x_{1}\left(  t\right) \\
ND^{\alpha}x_{2}\left(  t\right)  =x_{2}\left(  t\right)
\end{array}
\right.
\]

Thus we have to get $A_{1}$ and $N$. N'Doye \cite{NDoye2010Stabilization} has
proved that system $(E,A,\alpha)$ is regular if and only if $det(cE-A)$ is not
identically zero. Thus, $(cE-A)^{-1}$ exists. Define
\[
\hat{E}=(cE-A)^{-1}E,\  \hat{A}=(cE-A)^{-1}A
\]

Thus
\begin{align*}
\hat{A}  &  =(cE-A)^{-1}(cE+A-cE)\\
&  =c(cE-A)^{-1}E-I\\
&  =c\hat{E}-I
\end{align*}

According to standard Jordan matrix decomposition, there exists nonsingular
matrix $T$ such that
\[
T\hat{E}T^{-1}=diag(\hat{E}_{1},\hat{E}_{2})
\]
where $T\in \mathbb{R}^{n\times n}$; $\hat{E}_{1}\in \mathbb{R}^{n_{1}\times
n_{1}}$ is nonsingular; $\hat{E}_{2}\in \mathbb{R}^{n_{2}\times n_{2}}$ is a
nilpotent matrix. Thus, $c\hat{E}_{2}-I$ is nonsingular. Let
\begin{align*}
Q &=diag(\hat{E}_{1}^{-1},(c\hat{E}_{2}-I)^{-1})T(cE-A)^{-1}\\
P &=T^{-1}
\end{align*}

Then, we get
\begin{align*}
QEP  &  =diag(\hat{E}_{1}^{-1},(c\hat{E}_{2}-I)^{-1})T(cE-A)^{-1}ET^{-1}\\
&  =diag(I_{n_{1}},(c\hat{E}_{2}-I)^{-1}\hat{E}_{2})
\end{align*}
and
\begin{align*}
QAP  &  =diag(\hat{E}_{1}^{-1},(c\hat{E}_{2}-I)^{-1})T(cE-A)^{-1}AT^{-1}\\
&  =diag(cI_{n_{1}}-\hat{E}_{1}^{-1},I_{n_{2}})
\end{align*}

Therefore, we finally get $A_{1}$ and $N$
\[
A_{1}=\hat{E}_{1}^{-1}(c\hat{E}_{1}-I),\ N=(c\hat{E}_{2}-I)^{-1}\hat{E}_{2}
\]

Because $A_{1}$ is a constant matrix, by using Riemann-Liouville fractional
integral, we get (\cite{Luchko1999operational,Podlubny1999Fractional})
\[
x_{1}(t)-\underset{k=0}{\overset{m-1}{%
{\textstyle \sum}
}}x_{1}^{(k)}(0)\frac{t^{k}}{k!}=\frac{A_{1}}{\Gamma(\alpha)}\int
\limits_{0}^{t}x_{1}(\tau)(t-\tau)^{\alpha-1}d\tau
\]
where $m-1<\alpha \leq m$.

And according to Diethelm \cite{Diethelm2002predictor}
\[
\int_{0}^{t_{n+1}}(t_{n+1}-\tau)^{\alpha-1}x_{1}(t)d\tau \approx \frac
{z^{\alpha}}{\alpha(\alpha+1)}\underset{j=0}{\overset{n+1}{\sum}}%
a_{j,n+1}x_{1}(t_{j})
\]
where $z=t_{j+1}-t_{j}$ and

$a_{j,n+1}=\left \{  \begin{aligned}
&n^{\alpha+1}-(n-\alpha)(n+1)^{\alpha},&if\ j=0\\
&(n-j+2)^{\alpha+1}+(n-j)^{\alpha+1}-2(n-j+1)^{\alpha+1},&if\ 1\leq j\leq n\\
&1,&if\ j=n+1
\end{aligned}
\right.  $

In order to calculate $x_{1}(t_{n+1})$, Diethelm \cite{Diethelm2002predictor} predicts
the integral as%

\[
\int \nolimits_{0}^{t_{n+1}}(t_{n+1}-\tau)^{\alpha-1}x_{1}(\tau)d\tau
\approx \sum \limits_{j=0}^{n}b_{j,n+1}x_{1}(t_{j})
\]
where $b_{j,n+1}=\dfrac{z^{\alpha}}{\alpha}((n+1-j)^{\alpha}-(n-j)^{\alpha})$.
Thus, $x_{1}(t_{n+1})$ can be calculated by%

\[
x_{1}(t_{n+1})=\underset{k=0}{\overset{m-1}{%
{\textstyle \sum}
}}x_{1}^{(k)}(0)\dfrac{t^{k}}{k!}+\dfrac{z^{\alpha}}{\Gamma(\alpha+2)}%
A_{1}x_{1}^{p}(t_{n+1})+\dfrac{z^{\alpha}}{\Gamma(\alpha+2)}A_{1}%
\sum \limits_{j=0}^{n}a_{j,n+1}x_{1}(t_{j})
\]
where $x_{1}^{p}(t_{n+1})=\underset{k=0}{\overset{m-1}{%
{\textstyle \sum}
}}x_{1}^{(k)}(0)\dfrac{t^{k}}{k!}+\dfrac{1}{\Gamma(\alpha)}A_{1}%
\sum \limits_{j=0}^{n}b_{j,n+1}x_{1}(t_{j})$.

And according to solution (\ref{equx1x2}), $x_{2}\left(  t_{n}\right)  $ can
directly be calculated by
\[
x_{2}\left(  t_{n}\right)  =-\underset{k=1}{\overset{h-1}{%
{\textstyle \sum}
}}N^{k}\left(  \delta^{\left(  k\alpha-1\right)  }\left(  t_{n}\right)
x_{20}+\delta^{\left(  k\alpha-2\right)  }\left(  t_{n}\right)  x_{20}%
^{\left(  1\right)  }\right)
\]

Finally, we can get
\[
x(t_{n})=P\left[
\begin{array}
[c]{c}%
x_{1}(t_{n})\\
x_{2}\left(  t_{n}\right)
\end{array}
\right]
\]

\subsection{Numerical solution}

In this section, we verify the inequality (\ref{inequLMIun}) of theorem \ref{theoADM} as an example.

Consider system $(E,A,\alpha)$ with parameters $\alpha=1.8$, $A=\left[
\begin{array}
[c]{ccc}%
-1 & 0 & -1\\
0 & -2 & 0\\
0 & -1 & -1
\end{array}
\right]  $, $E=\left[
\begin{array}
[c]{ccc}%
1 & 0 & 0\\
1 & 1 & -1\\
0 & 0 & 0
\end{array}
\right]  $. And $E_{0}=\left[
\begin{array}
[c]{c}%
0\\
0\\
1
\end{array}
\right]  $ can be chosen to satisy $E^{T}E_{0}=0.$

Then, by solving LMI (\ref{inequLMIun}), we get %

\[
P=\left[
\begin{array}
[c]{ccc}%
1.7896 & -0.2755 & -0.5029\\
-0.2755 & 0.8271 & -0.6667\\
-0.5029 & -0.6667 & 1.5113
\end{array}
\right]  ,\text{ }Q=[%
\begin{array}
[c]{ccc}%
-0.043 & 0.3709 & 0.3849
\end{array}
]
\]

It means the system is admissible. Eigenvalues of the system is shown in figure \ref{FigEig}. From figure \ref{FigEig}, we can find that all the
eigenvalues of $(E,A)$ lie in the stable area. State response of the system is shown in figure \ref{FigCurve}, which implies that
the system is stable.

\begin{figure}
\centering
  \includegraphics[width=70mm]{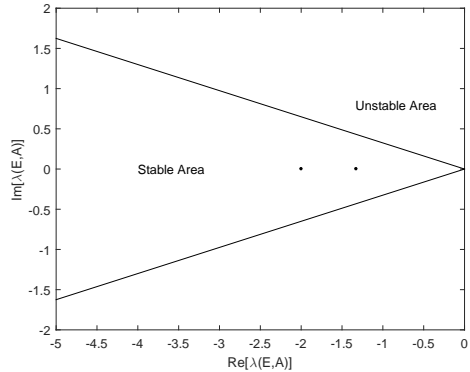}\\
  \caption{Eigenvalues of the system}\label{FigEig}
\end{figure}
\begin{figure}
  \centering
  \includegraphics[width=140mm]{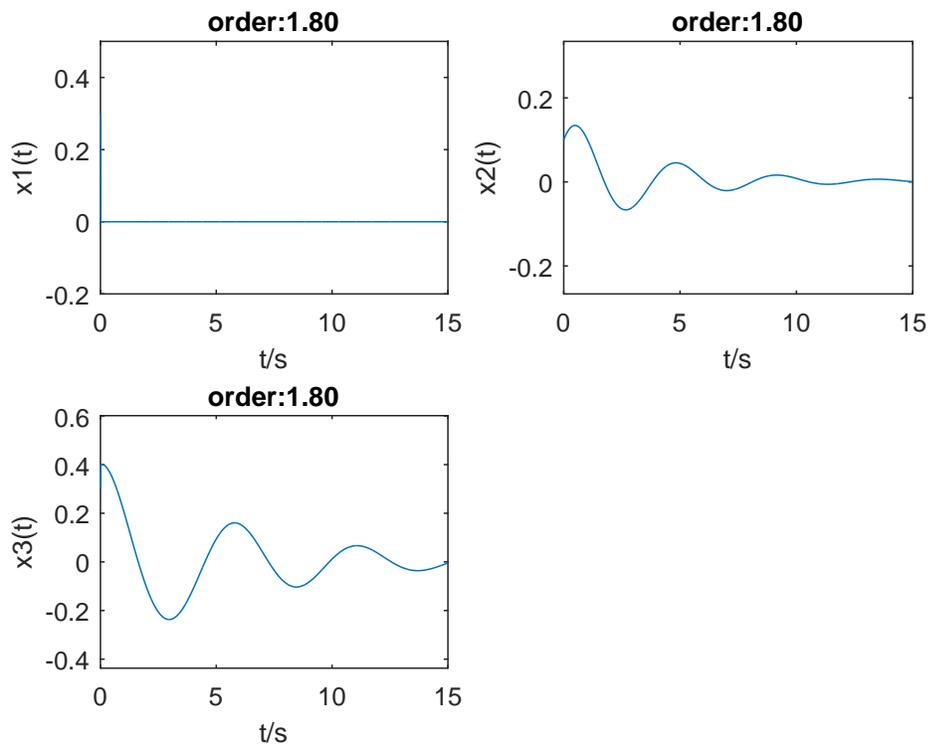}\\
  \caption{State response of the system}\label{FigCurve}
\end{figure}

\section{Conclusion}

In this paper, singular fractional order system with fractional order $1<\alpha<2$
has been studied. The regularity and impulse-free of SFOS are proved in time domain.
 Then, this paper analysed sufficient and necessary conditions of stability and admissibility, respectively.
 After that, sufficient conditions of robust admissibility were given. Finally, numerical example was illustrated
 to verify proposed theorem.


\bibliographystyle{unsrt}
\bibliography{Generalized-Fractional-Order-System}

\end{document}